\documentclass[11pt]{amsart}

\usepackage{amssymb}
\usepackage{amsmath}
\usepackage[dvips]{graphicx}
\usepackage{enumitem}

\usepackage{verbatim}

\newtheorem{theorem}{Theorem}[section]
\newtheorem{lemma}[theorem]{Lemma}

\theoremstyle{definition}
\newtheorem{definition}[theorem]{Definition}
\newtheorem{example}[theorem]{Example}

\theoremstyle{remark}
\newtheorem{remark}[theorem]{Remark}

\numberwithin{equation}{section}

\begin{document}

\title[Extrinsic Ricci Flow on Surfaces of Revolution]{Extrinsic Representation of Ricci Flow on Surfaces of Revolution}
\author{Vincent E.\ Coll Jr., Jeff Dodd, and David L. Johnson}
\date{November 1, 2013}

\begin{abstract}
An extrinsic representation of a Ricci flow on a differentiable $n$-manifold
$M$ is a family of
submanifolds $S(t)$, each smoothly embedded in $\mathbb{R}^{n+k}$,
evolving as a function of time $t$ such that the metrics induced on the submanifolds
$S(t)$ by the ambient Euclidean metric yield
the Ricci flow on $M$.  When does such a representation exist?

We formulate this question precisely and describe
a new, comprehensive way of addressing it for surfaces of revolution in $\mathbb{R}^3$.  Our approach is to
build the desired embedded surfaces of revolution $S(t)$ in $\mathbb{R}^3$ into the flow at the outset
by rewriting the Ricci flow equations in terms of extrinsic geometric quantities in a
natural way.  This identifies an extrinsic representation with a
particular solution of the scalar logarithmic diffusion equation in one space variable.
The result is a single, unified framework to construct an extrinsic representation in $\mathbb{R}^3$  of a Ricci flow
on a surface of revolution $S$ initialized by a metric $g_0$.

Of special interest is the Ricci flow on the torus $S^1 \times S^1$  embedded in $\mathbb{R}^3$.  In this case,
the extrinsic representation of the Ricci flow on a Riemannian cover of $S$ is eternal.  This flow can also be realized as a compact family of nonsmooth, but isometric, embeddings of the torus into
$\mathbb{R}^3$.
\end{abstract}

\maketitle

\section{Introduction}

Ricci flow is an intrinsic geometric flow: a metric $g(t)$ evolving as
a function of time $t$ on a fixed differentiable manifold $M$ according to the Ricci flow equation $\partial_t g_{ij}(t) = 2R_{ij}(t)$.
\begin{definition} \label{extrinsicrepresentation}
\mbox{}
\begin{itemize}[leftmargin=20pt]
\item[(a)] A {\em local extrinsic representation} of a Ricci flow $(M, g(t))$ on a
differentiable $n$-manifold $M$ in $\mathbb{R}^{n+k}$ consists of a family of $n$-dimensional submanifolds $S(t)$ that are smoothly embedded in $\mathbb{R}^{n+k}$
together with corresponding local isometries $i(t): (S(t), g_E(t)) \rightarrow (M, g(t))$, where $g_E(t)$
denotes the metric induced on $S(t)$ by the ambient Euclidean metric on $\mathbb{R}^{n+k}$.
\item[(b)] A {\em global extrinsic representation} of a Ricci flow $(M, g(t))$ on a differentiable $n$-manifold $M$
is a local extrinsic representation such that the local isometries $i(t)$ are global isometries.
\end{itemize}
\end{definition}
J.\ H.\ Rubinstein and R.\ Sinclair~\cite{RubinsteinandSinclair2005} have shown that on the sphere $S^n$, for $n \geq 2$, there exists
a global extrinsic representation in $\mathbb{R}^{n + 1}$  of any Ricci flow that is initialized by a metric $g_0$ such that $(S^n, g_0)$
can be isometrically embedded in $\mathbb{R}^{n + 1}$ as a hypersurface of revolution.  J.\ C.\ Taft~\cite{Taft2010} has shown that steady and expanding
Ricci solitons in all dimensions admit global extrinsic representations.
Here we present a new point of view:  a complete analytic model for the problem of constructing extrinsic representations of Ricci
flows on surfaces of revolution in $\mathbb{R}^3$.

We consider Ricci flows initialized by surfaces of revolution that are smoothly immersed in $\mathbb{R}^3$, connected, complete, and
without boundary, and we treat these flows differently depending upon whether or not the initial surface of revolution is {\em toroidal}, that is,
generated by a closed profile curve that does not intersect the axis of revolution.
In particular, we show that any global extrinsic representation in $\mathbb{R}^3$ of a Ricci flow initialized by a non-toroidal surface
of revolution can be reformulated as a solution of the scalar logarithmic diffusion equation in one space variable satisfying
certain conditions, and we illustrate this new point of view by applying it to the extrinsic representations constructed by Rubinstein and Sinclair for
Ricci flows on the $2$-sphere.
We also find that, conversely, a global extrinsic representation in $\mathbb{R}^3$ of a Ricci flow initialized by a non-toroidal surface of revolution
can be constructed if there exists a solution of the scalar logarithmic diffusion equation in one space variable meeting certain conditions.  We
demonstrate this new method of construction by producing new extrinsic representations of Ricci flows initialized by a large class of
non-toroidal surfaces of revolution.

Finally, we modify this construction to produce a local
extrinsic representation of any Ricci flow initialized by a smoothly immersed toroidal surface of revolution.  Such a flow becomes a global extrinsic representation in $\mathbb{R}^3$ on a cylindrical Riemannian cover of the torus initialized by an isometric embedding in $\mathbb{R}^3$ as a periodic surface of revolution.  This flow can also be realized as a non-smooth but isometrically-embedded family of tori.  As far as we know, these are the first extrinsic representations of Ricci flows on the torus.

\section{Global Extrinsic Representations of Ricci Flows Initialized by Non-toroidal Surfaces of Revolution}

On $2$-manifolds, Ricci flow is conformal; that is, each metric $g_t$ is conformally equivalent to the initial metric $g_0$~\cite{Brendle2006}.  We begin with an observation that applies not only to extrinsic representations
of Ricci flows on surfaces of revolution in $\mathbb{R}^3$, but more generally to all conformal extrinsic flows of surfaces of revolution.

\begin{definition} A point $P$ of a surface of revolution $S$ is a {\em pole} if $P$ is an intersection point of $S$ with the axis of revolution.
\end{definition}

\begin{lemma}[Reduction of Conformal Extrinsic Flow to Scalar Flow in Isothermal Coordinates] \label{isothermalcoordinates}
Let $S_0$ be a surface of revolution that is connected, complete, and without boundary, and smoothly immersed in $\mathbb{R}^3$ by the
parametrization
\begin{align} \label{initialsurface}
  x &= f_0(v) \cos \theta  \\
  y &= f_0(v) \sin \theta  \notag \\
  z &= h_0(v) \notag
\end{align}
where $v \in I$ for some interval $I \subseteqq \mathbb{R}$, $f_0(v) \geq 0$ for all $v \in I$, and $\theta \in [0,2\pi]$.
Consider a conformal extrinsic flow of smooth immersed surfaces $S(t)$ parameterized for $t \geq 0$ by:
\begin{align} \label{extrinsicflow}
  x &= f(v,t) \cos \theta   \\
  y &= f(v,t) \sin \theta  \notag   \\
  z &= h(v,t) \notag
\end{align}
where $f$ and $h$ are smooth functions such that $f(v,0) = f_0(v)$ and $h(v,0) = h_0(v)$.
Without loss of generality, assume that, at a pole, where $f_0(p) = 0$, then $f(p,t) = 0$ for $t > 0$.

For $t \geq 0$, reparameterize $S(t)$ with coordinates $(\xi, \theta)$ where
\begin{equation} \label{xidefinition}
\xi (v) = \int_{q}^v \frac{ \sqrt{f_0^{\prime}(s)^2 + h_0^{\prime}(s)^2}} {f_0(s)} \mbox{ d}s
\end{equation}
for some $q \in I$ such that $f_0(q) \neq 0$.  Then the reparameterized extrinsic geometric flow
\begin{align} \label{reparameterizedflow}
  x &= f(\xi,t) \cos \theta   \\
  y &= f(\xi,t) \sin \theta  \notag   \\
  z &= h(\xi,t) \notag
\end{align}
satisfies
\begin{equation}
f_{\xi}^2 + h_{\xi}^2 = f^2  \label{arclengthcondition}
\end{equation}
and the corresponding reparameterized metric flow induced by the ambient Euclidean metric is given by
\begin{equation} \label{isothermalmetricflow}
g_E(\xi,t) =  f(\xi,t)^2 \left[ {\rm d}\xi^2 + {\rm d}\theta^2 \right] .
\end{equation}
A pole location $v = p$ for the flow corresponds to $\xi = -\infty$ or $\xi = \infty$.
The map from $I \rightarrow \mathbb{R}$ given by $v \mapsto \xi(v)$ has a range of $(-\infty, \infty)$.
\end{lemma}
\begin{proof}
The metric on $S(t)$ induced by the ambient Euclidean metric is:
\begin{equation} \nonumber
g_E(v,t) = (f_v^2 + h_v^2) {\rm d}v^2 + f^2 {\rm d}\theta^2.
\end{equation}
At any time $t$, isothermal coordinates on $S(t)$ are given by
$(\xi, \theta)$ where $\xi$ and $v$ are related by the differential equation
\begin{equation} \nonumber
\frac{{\rm d} v}{{\rm d} \xi} = \frac{f}{\sqrt{f_v^2 + h_v^2}} .
\end{equation}
But because the flow is conformal, we have for all $t > 0$ that
\begin{equation} \nonumber
\frac{f(v,t)}{\sqrt{f_v(v,t)^2 + h_v(v,t)^2}} = \frac{f_0(v)}{\sqrt{f_0^{\prime}(v)^2 + h_0^{\prime}(v)^2}} .
\end{equation}
Thus the reparameterization \eqref{xidefinition} renders $S(t)$ in isothermal coordinates for all $t \geq 0$.
In these coordinates, the metric flow is given by \eqref{isothermalmetricflow}, and comparing arclength elements in
\eqref{reparameterizedflow} and \eqref{isothermalmetricflow} leads to \eqref{arclengthcondition}.

If $S_0$ has a pole at $v = p$, then  $f_0$ must extend to an odd function of $(v-p)$ in a neighborhood of the pole, and it follows by a theorem of H.\ Whitney ~\cite{Whitney1943} that, if the embedding is of $C^{2k + 1}$ smoothness at the pole, then  $f_0$ can be written as
\begin{equation} \label{polestructure}
f_0(v) = (v - p)F_0 \left( (v - p)^2 \right)
\end{equation}
where $F_0(0) \neq 0$ and $F_0$ is of $C^k$ smoothness at $0$.  It follows that
\begin{equation} \nonumber
\lim_{v \rightarrow p} \xi(v) = \int_{q}^p \frac{ \sqrt{f_0^{\prime}(s)^2 + h_0^{\prime}(s)^2}} {f_0(s)} \mbox{ d}s =
\left\{
  \begin{array}{rl}
    -\infty, & \hbox{if $p < q$;} \\
    \infty, & \hbox{if $p > q$.}
  \end{array}
\right.
\end{equation}

When $S_0$ is non-toroidal, the statement regarding the range of $\xi$ is a straightforward consequence of
its definition \eqref{xidefinition} and the behavior of $\xi$ at a pole.
When $S_0$ is toroidal, we will take $f_0(v)$ and $h_0(v)$ to be smooth periodic functions defined for $-\infty < v < \infty$
with a common period in $v$ such that $f_0(v) > 0$ for all $v$, so that $f_0(\xi)$ and $h_0(\xi)$ are smooth
periodic functions defined for $-\infty < \xi < \infty$ with a common period in $\xi$ such that $f_0(\xi) > 0$ for all $\xi$.
\end{proof}
\begin{definition}
We refer to  the special coordinates $(\xi, \theta)$ constructed for the conformal extrinsic metric flow \eqref{extrinsicflow} in
Lemma~\ref{isothermalcoordinates} as {\em time-independent isothermal coordinates}.
\end{definition}
On a two-dimensional Riemannian manifold $M$, Ricci flow takes the form of the initial value problem
\begin{align} \label{RicciFlow}
\partial_t g(t) &=  -2 K(t) g(t), \quad \quad 0 < t < T  \\
g(0) &=  g_0  \notag
\end{align}
where $g(t)$ is the metric on $M$ at time $t$ and $K(t)$ is the Gaussian curvature associated with the metric $g(t)$.
It is well known (see the survey by J.\ Isenberg, R.\ Mazzeo, and N.\ Sesum  ~\cite{Isenberg2011}) that
\eqref{RicciFlow} can be reformulated as a scalar logarithmic diffusion equation.  In particular, if $g_0^{\ast}$
is any metric in the conformal class of the initial metric $g_0$, then \eqref{RicciFlow} is equivalent to the scalar flow
$g(t) = u(t) g_0^{\ast}$ given by
\begin{align} \label{abstractlogdiffusion}
\partial_t u &=  \triangle_{g_0^{\ast}} \left[ \log u \right] - 2 K_{g_0^{\ast}}, \quad \quad 0 < t < T   \\
g_0 &=  u(0) g_0^{\ast} \notag
\end{align}
where $\triangle_{g_0^{\ast}}$ is the Laplace-Beltrami operator and $K_{g_0^{\ast}}$ the Gaussian curvature associated
with the metric $g_0^{\ast}$.

Let $S_0$ be a surface of revolution smoothly immersed in $\mathbb{R}^3$ by the parameterization \eqref{initialsurface},
and let $(\xi,\theta)$ be isothermal coordinates on $S_0$ as in \eqref{xidefinition}.  Writing \eqref{abstractlogdiffusion} for
$g_0^{\ast} = {\rm d}\xi^2 + {\rm d}\theta^2$, we have
$K_0^{\ast} = 0$ and $\triangle_{g_0^{\ast}} = \partial^2 / \partial \xi^2 + \partial^2 / \partial^2 \theta^2$.
So intrinsically, a Ricci flow initialized by $S_0$ is given by
\begin{equation} \label{intrinsicmetricflow}
 g(t) = u(\xi,t) \left[ {\rm d}\xi^2 + {\rm d}\theta^2 \right]
\end{equation}
where $u$ satisfies the logarithmic diffusion equation in one space variable:
\begin{align} \label{intrinsiclogdiffusion}
 u_t(\xi,t) &= \left[ \log u(\xi,t) \right]_{{\xi}{\xi}} & & \mbox{for $\xi$ in $\mathbb{R}$, $0 < t < T$}  \\
 u(\xi,0) &= u_0(\xi)  & & \mbox{for $\xi$ in $\mathbb{R}$} \notag
\end{align}
with $u_0(\xi) = f_0(\xi)^2$.

To characterize global extrinsic representations in $\mathbb{R}^3$ of Ricci flows initialized by non-toroidal surfaces of revolution,
we apply Lemma~\ref{isothermalcoordinates} to the intrinsic formulation \eqref{intrinsiclogdiffusion}.
\begin{theorem}  \label{representationtheorem}
Let $S_0$ be a surface of revolution smoothly immersed in $\mathbb{R}^3$ by the parameterization \eqref{initialsurface} that is connected,
complete, without boundary, and non-toroidal.  Let $(\xi,\theta)$ be isothermal coordinates on $S_0$ as in \eqref{xidefinition}.
\begin{itemize}[leftmargin=20pt]
\item[(a)] Suppose that a Ricci flow $(S_0, g(t))$ initialized by $S_0$ has a global extrinsic representation in $\mathbb{R}^3$  for $0 \leq t < T$.
Then the each embedded surface $S(t)$ in this representation can be parameterized by \eqref{reparameterizedflow}
where the function $u(\xi,t) = f(\xi,t)^2$ is a smooth positive solution of the initial value problem \eqref{intrinsiclogdiffusion}
such that for each $t \in [0,T)$ the following conditions hold:
\begin{itemize}
\item[(1)] (embeddability of the metric) $\displaystyle{\underset{\xi \in \mathbb{R}}{\sup} \left| \dfrac{f_{\xi}}{f} \right| \leq 1}$, and
\item[(2)] (smoothness at poles) at a pole location $v = p$ of $S_0$ where $\xi \rightarrow -\infty$ and $v \rightarrow p^+$,
$f_{\xi}/f \rightarrow 1$ as $\xi \rightarrow -\infty$, and at a pole location
$v = p$ of $S_0$ where $\xi \rightarrow \infty$ and $v \rightarrow p^-$,
$f_{\xi}/f \rightarrow -1$ as $\xi \rightarrow \infty$.
\end{itemize}
\item[(b)] Conversely, suppose that there is a smooth positive solution $u(\xi,t)$ of
the initial value problem \eqref{intrinsiclogdiffusion} with $u_0(\xi) = f_0(\xi)^2$ such that for $f(\xi,t) = \sqrt{u(\xi,t)}$, conditions~(1) and~(2)
hold for each $t \in [0,T)$.  Then there exists a Ricci flow $(S_0, g(t))$ initialized by $S_0$ that has a global extrinsic representation in $\mathbb{R}^3$ for $0 \leq t < T$
comprised of embedded surfaces $\widehat{S}(t)$ parameterized by
\begin{align} \label{constructedextrinsicflow}
  x &= f(\xi,t) \cos \theta   \\
  y &= f(\xi,t) \sin \theta  \notag   \\
  z &= \widehat{h}(\xi,t) \notag
\end{align}
where $\widehat{h}(\xi,t) = \int \sqrt{f(\xi,t)^2 - f_{\xi}(\xi,t)^2} \mbox{ d}\xi$ for $\xi \in \mathbb{R} $ and $\theta \in [0,2\pi]$.
The surfaces $\widehat{S}(t)$ are of at least $C^1$ smoothness.  For each $t \in [0,T)$, $\widehat{h}(\xi,t)$ is a non-decreasing
function of $\xi$.  In particular, if $h_0$ is not monotone, then $\widehat{h}(\xi,0) \neq h_0(\xi)$, and $\widehat{S}(0) \neq S_0$.
\end{itemize}
\end{theorem}

\begin{remark}
In part (b), although it is not necessarily true that $\widehat{S}(0) = S_0$ (because for the extrinsic representation
the function $\widehat{h}$ is constructed to be monotone), these two surfaces are locally isometric, and if $S_0$ is non-toroidal,
then $\widehat{S}(0)$ and $S_0$ are globally isometric, although not necessarily by an ambient isometry.
If $S_0$ is topologically a torus, then $\widehat{S}(0)$ will be a Riemannian cover of $S_0$.
\end{remark}

\begin{proof}
\mbox{ }

\smallskip

{\em Part (a)}:  Because Ricci flow preserves isometries, any extrinsic representation of a Ricci flow in $\mathbb{R}^3$ initialized by $S_0$ must be comprised of
embedded surfaces of revolution $S(t)$, which we can take to be parameterized as in \eqref{extrinsicflow} for times $0 < t < T$.
Moreover, there is an
isometry $i(0): S(0) \rightarrow S_0$, so isothermal coordinates $(\xi, \theta)$ on $S_0$ also serve as isothermal coordinates
on $S(0)$.  By Lemma~\ref{isothermalcoordinates}, $(\xi, \theta)$ constitute time-independent isothermal coordinates for the
extrinsic flow $S(t)$ for $t \geq 0$.

For each time $t$ there exists an isometry $i(t): (S(t),g_E(t)) \rightarrow (S_0, g(t))$.  Comparing the extrinsic metric
$g_E(t)$ given by \eqref{isothermalmetricflow} with the intrinsic metric $g(t)$ given by \eqref{intrinsicmetricflow}, we
see that $u(\xi,t) = f(\xi,t)^2$ is a smooth solution of the initial value problem \eqref{intrinsiclogdiffusion}
with $u_0(\xi) = f_0(\xi)^2$.  Condition~(1) follows from
the arclength equation \eqref{arclengthcondition};  because the curve $\xi \mapsto (f(\xi,t), h(\xi,t))$
is immersed in $\mathbb{R}^2$, $f(\xi,t)^2 - f_{\xi}(\xi,t)^2 \geq 0$.
Condition~(2) follows by noting that $v=p$ is a pole location for $S(t)$ where $S(t)$ is of
at least $C^1$ smoothness.
The indeterminate limits in question can be resolved using the structure of $f_0$ at a pole location $v = p$ given
by \eqref{polestructure}, and the analogous structure for $t > 0$:  if $f(v,t)$ is of $C^{2k + 1}$ smoothness in $v$ at $v = p$, then
\begin{equation} \nonumber
f(v,t) = (v - p)F((v - p)^2,t)
\end{equation}
where  $F(0,t) \neq 0$, and $F(w,t)$ is of $C^k$ smoothness in $w$ at $w = 0$.  A calculation shows that for $t \geq 0$:
\begin{equation} \label{fluxatpole}
\lim_{\xi \rightarrow \pm \infty} \frac{f_{\xi}}{f} = \lim_{v \rightarrow p} \frac{f_v}{f} \frac{{\rm d}v}{{\rm d}\xi}
 = \frac{f_0^{\prime}(p)}{|f_0^{\prime}(p)|}.
\end{equation}
If $\xi \rightarrow -\infty$, then $v \rightarrow p^{+}$ and $f \rightarrow 0$, so $f_0^{\prime}(p) > 0$ and the limit \eqref{fluxatpole}
is $1$, whereas if $\xi \rightarrow \infty$, then $v \rightarrow p^{-}$ and $f \rightarrow 0$, so $f_0^{\prime}(p) < 0$ and the limit
\eqref{fluxatpole} is $-1$.

\smallskip

{\em Part (b)}:  By condition~(1), we can define for $\xi \in \mathbb{R}$ and $0 \leq t \leq T$ a function
$$ \widehat{h}(\xi,t) = \int \sqrt{f(\xi,t)^2 - f_{\xi}(\xi,t)^2} \mbox{ d}\xi$$
of at least $C^1$ smoothness and, using $f(\xi,t)$ and $\widehat{h}(\xi,t)$,
the family of embedded surfaces \eqref{constructedextrinsicflow} for $0 \leq t < T$.
The metric on $\widehat{S}(t)$ is given by
$$ \widehat{g}_E(t) = (f_{\xi}(\xi,t)^2 + h_{\xi}(\xi,t)^2) \mbox{ d}\xi^2 + f(\xi,t)^2 \mbox{ d}\theta^2  = f(\xi,t)^2 ({\rm d}\xi^2 + {\rm d}\theta^2) $$
where $u(\xi,t) = f(\xi,t)^2$ is a solution of the initial value problem \eqref{intrinsiclogdiffusion}
with $u_0(\xi) = f_0(\xi)^2$.  If $h_0$ is not monotone then
$S_0 \neq \widehat{S}(0)$, but nevertheless the coordinates $(\xi, \theta)$ serve as time-independent isothermal coordinates
for the extrinsic flow $\widehat{S}(t)$.  For each time $t$, identifying each point having coordinates $(\xi, \theta)$ on $(\widehat{S}(t), \widehat{g}_E(t))$
with the corresponding point having the same coordinates $(\xi, \theta)$ on $(S_0, g(t))$ and comparing the extrinsic metric
$\widehat{g}_E(t)$ with the intrinsic metric $g(t)$ as given by \eqref{intrinsicmetricflow} yields an isometry
$i(t): (\widehat{S}(t), g_E(t))  \rightarrow (S_0, g(t))$.  That is, the embedded surfaces $\widehat{S}(t)$ comprise a global extrinsic
representation of a Ricci flow initialized by $S_0$.

Suppose now that $\xi \rightarrow \infty$ or $\xi \rightarrow -\infty$ at a pole location $v = p$ of $S_0$, and that $t \in (0,T)$.
A.\ Rodr{\'{\i}}guez and J.\ L.\ V{\'a}zquez \cite{Rodriguez1990,Rodriguez1995} have shown that condition~(2) ensures that $f(\xi,t) \sim {\rm e}^{-c|\xi|}$.
Thus $f(\xi,t) \rightarrow 0$.  Moreover, as $f(\xi,t) \rightarrow 0$,
the arclength element on the curve $(x,z) = (f(\xi,t), \widehat{h}(\xi,t))$ is $f(\xi,t)$, which is integrable, resulting in a corresponding pole of $\widehat{S}(t)$.
That $\widehat{S}(t)$ is $C^1$ smooth at such a pole also follows since as $f(\xi,t) \rightarrow 0$,
$\displaystyle{\frac{{\rm d}z}{{\rm d}x} = \frac{\widehat{h}_{\xi}}{f_{\xi}} = \frac{\sqrt{f^2 - f_{\xi}^2}}{f_{\xi}} = \sqrt{\frac{f^2}{f_\xi^2} - 1} \rightarrow 0}$.
\end{proof}
Part (a) of Theorem~\ref{representationtheorem} provides an interpretation of global extrinsic representations in $\mathbb{R}^3$
of Ricci flows on surfaces of revolution
in terms of the dynamics of logarithmic diffusion.  Consider, for example, the global extrinsic representations of Ricci
flows on the $2$-sphere constructed by Rubinstein and Sinclair ~\cite{RubinsteinandSinclair2005}.  In this case, the initial
surface $S_0$, parameterized as in \eqref{initialsurface}, has poles at two values of the
original parameter $v$, say $v = p_1$ and $v = p_2$ where $p_1 < p_2$.  In time-independent isothermal coordinates, $\xi \rightarrow -\infty$ as
$v \rightarrow p_1^+$ and $\xi \rightarrow \infty$ as $v \rightarrow p_2^-$, and $u(\xi,t) = f(\xi,t)^2$ satisfies the
initial value problem \eqref{intrinsiclogdiffusion} with $u_0(\xi,t) = f_0(\xi)^2$.

The initial value problem
\eqref{intrinsiclogdiffusion} has multiple solutions for any initial function $u_0 \in L^1_{{\rm loc}}(\mathbb{R})$, but if
$u_0$ is a positive $L^1$ function, a unique solution can be
specified by imposing appropriate conditions on the flux $u_{\xi}/u$ associated with the conservation law in \eqref{intrinsiclogdiffusion}
at $x = -\infty$ and $x = \infty$ \cite{Rodriguez1990,Rodriguez1995}.
In particular, imposing conditions of the form
$u_{\xi}/u \rightarrow a$ as $\xi \rightarrow -\infty$ and $u_{\xi}/u \rightarrow -b$ as $\xi \rightarrow \infty$
for positive constants $a$ and $b$ uniquely determines a smooth, positive solution of \eqref{intrinsiclogdiffusion} for which
$$ \int_{-\infty}^{\infty} u(\xi,t) \mbox{ d}\xi =  \int_{-\infty}^{\infty} u_0(\xi) \mbox{ d}\xi - (a + b)t. $$
The conditions $f_{\xi}/f \rightarrow 1$ as $\xi \rightarrow -\infty$ and $f_{\xi}/f \rightarrow -1$ as $\xi \rightarrow \infty$
that maintain smoothness at the poles of $S(t)$ for $t > 0$ in Theorem~\ref{representationtheorem} are equivalent to the flux conditions
\begin{equation} \label{fluxconditions}
\frac{u_{\xi}}{u} = \frac{2f_{\xi}}{f} \rightarrow 2 \quad \mbox{as $\xi \rightarrow -\infty$}, \quad \quad \frac{u_{\xi}}{u} =
\frac{2f_{\xi}}{f} \rightarrow -2 \quad \mbox{as $\xi \rightarrow \infty$}
\end{equation}
Since $\int_{-\infty}^{\infty} u(\xi,t) \mbox{ d}t = \int_{-\infty}^{\infty} f(\xi,t)^2 \mbox{ d}\xi = (1/ 2\pi) \left( \mbox{surface area of $S(t)$} \right)$,
the flux conditions \eqref{fluxconditions} impose at each pole of $S(t)$ a loss of surface area at a constant rate of $2\pi$ square units
per unit time, resulting in the well-known extinction time of $T = (1/ 8\pi) \left( \mbox{surface area of $S_0$} \right)$.  The identification
of $\xi = \pm \infty$ with the poles of $S(t)$ in this flow is an extrinsic geometric realization of the intuition expressed by
J.\ L.\ V{\'a}zquez that for $f(\xi,t)$ near 0, ``[logarithmic] diffusion is so fast that, in some sense, infinity lies at a
finite distance $\ldots$'' ~\cite[p. 141]{Vazquez2006}.
\begin{example}[{\bf The shrinking sphere}]
The Ricci flow on the unit $2$-sphere initialized by the canonical metric has a well-known explicit extrinsic representation in $\mathbb{R}^3$ in which $S(t)$
is usually parameterized by \eqref{extrinsicflow}
with $f(v,t) = (\sqrt{1 - 2t}) \cos v$ and $h(v,t) = (\sqrt{1 - 2t}) \sin v$ for $-\pi/2 \leq v \leq \pi/2$ and $0 \leq t \leq 1/2$.
In time-independent isothermal coordinates, $S(t)$ is parameterized by \eqref{reparameterizedflow} with $f(\xi,t) = (\sqrt{1 - 2t}) \mbox{ {\rm sech} } \xi$
and $h(\xi,t) = (\sqrt{1 - 2t}) \tanh \xi$ for $-\infty < \xi < \infty$ and $0 \leq t \leq 1/2$.
\end{example}
Part (b) of Theorem~\ref{representationtheorem} provides a method of constructing a global extrinsic representation in $\mathbb{R}^3$ of a Ricci
flow initialized by a non-toroidal surface of revolution, provided that
the required solution of the logarithmic diffusion equation exists.  Here we illustrate this construction
using the following facts about logarithmic diffusion in one space dimension:
\begin{lemma}[J.\ R.\ Esteban, A.\ Rodr{\'{\i}}guez and J.\ L.\ V{\'a}zquez ~\cite{Esteban1988}] \label{existencetheorem}
For any initial function $u_0(\xi)$ having bounded derivatives of all
orders and such that for all $\xi \in \mathbb{R}$, $0 < m \leq u_0(\xi) \leq M$ for some positive numbers $m$ and $M$, the initial value problem
\eqref{intrinsiclogdiffusion} has a unique solution $u(\xi,t) \in C^{\infty}(\mathbb{R} \times (0,\infty))$ satisfying the conditions
$0 < m \leq u(\xi,t) \leq M$ and $|u_{\xi}(\xi,t)/u(\xi,t)| \leq \sup_{\xi \in \mathbb{R}} | u_0^{\prime}(\xi)/u_0(\xi) |$ for all $(\xi,t) \in \mathbb{R} \times (0,\infty)$.
For every $\xi \in \mathbb{R}$ and $t > 0$,
\begin{equation} \label{estimate}
\left( \frac{1}{u(\xi,t)} \right)_{\xi \xi}  < \frac{1}{t}.
\end{equation}
\end{lemma}
Lemma~\ref{existencetheorem} quickly yields eternal global extrinsic representations of certain Ricci flows initialized by suitable immersed surfaces of revolution:
\begin{theorem} \label{easynopolecase}
Let $S_0$ be a surface of revolution smoothly immersed in $\mathbb{R}^3$ by the parameterization \eqref{initialsurface} that is connected,
complete, without boundary, and non-toroidal.
Let $(\xi, \theta)$ be isothermal coordinates on $S_0$ as in \eqref{xidefinition}.
Suppose further that the radius function $f_0(\xi)$ has bounded derivatives of all orders
and that for some positive numbers $m$ and $M$, $0 < m \leq f_0(\xi) \leq M$ for all $\xi \in \mathbb{R}$.  Then there exists a Ricci flow initialized by $S_0$ that has a global extrinsic
representation  in $\mathbb{R}^3$ for $0 < t < \infty$.
\end{theorem}
\begin{proof}
Let $u(\xi,t)$ be the solution of the initial value problem \eqref{intrinsiclogdiffusion} given by Lemma~\ref{existencetheorem}
for the initial function $u_0(\xi) = f_0(\xi)^2$.
Because $S_0$ has no poles, the theorem follows from part (b) of Theorem~\ref{representationtheorem} if we can verify condition~(1).  This condition holds because for any $\xi \in \mathbb{R}$ and $t > 0$:
$$ \left| \frac{f_{\xi}}{f} \right| = \frac{1}{2} \left| \frac{u_{\xi}}{u} \right| \leq \frac{1}{2} \sup_{\xi \in \mathbb{R}} \left| \frac{u_0^{\prime}(\xi)}{u_0(\xi)} \right| =
\sup_{\xi \in \mathbb{R}} \left| \frac{f_0^{\prime}(\xi)}{f_0(\xi)} \right| \leq 1. $$
(The last inequality follows from the arclength equation \eqref{arclengthcondition} because the curve $\xi \mapsto (f_0(\xi), h_0(\xi))$
is immersed in $\mathbb{R}^2$, so that $f_0(\xi)^2 - f_0^{\prime}(\xi)^2 \geq 0$ for all $\xi \in \mathbb{R}$.)
\end{proof}

\section{Local Extrinsic Representations of Ricci Flows Initialized by Toroidal Surfaces of Revolution}

Construction of a global extrinsic representation in $\mathbb{R}^3$ for the Ricci flow of a toroidal surface of revolution $S_0$ is problematic for topological reasons, and may be impossible~\cite{Taft2010}.
Nevertheless,
Theorem~\ref{easynopolecase} can be used to construct a {\em local} extrinsic representation in $\mathbb{R}^3$ of any such flow.
\begin{theorem} \label{periodicinitialcurve}
Let $S_0$ be a toroidal surface of revolution immersed in $\mathbb{R}^3$ by the parameterization \eqref{initialsurface} where $f_0$ and $h_0$ are smooth periodic functions
defined for $-\infty < v < \infty$
with a common period $P$ in $v$ such that $f_0(v) > 0$ for all $v$.  Let $Q$ be the common period of $f_0(\xi)$ and $h_0(\xi)$ in isothermal coordinates $(\xi, \theta)$:
\begin{equation} \nonumber
Q =  \int_0^P \frac{ \sqrt{f_0^{\prime}(s)^2 + h_0^{\prime}(s)^2}} {f_0(s)} \mbox{ d}s.
\end{equation}
Then for times $0 \leq t < \infty$, the Ricci flow initialized by $S_0$ has a local extrinsic representation  in $\mathbb{R}^3$ comprised of surfaces of the form
\begin{equation} \nonumber
\widehat{S}(t) = \left\{ \left( f(\xi,t) \cos \theta, f(\xi,t) \sin \theta, \widehat{h}(\xi,t) \right): 0 \leq \theta \leq 2\pi, -\infty < \xi < \infty \right\}.
\end{equation}
For each $t \geq 0$, $\widehat{h}(\xi,t)$ is monotone in $\xi$, but $f(\xi,t)$ and $h_{\xi}(\xi,t)$ each have period $Q$ in $\xi$, so there exists a positive function $Z(t)$ such that for all $-\infty < \xi < \infty$, $\widehat{h}(\xi+Q,t)= \widehat{h}(\xi,t)+Z(t)$.  Thus the surfaces $\widehat{S}(t)$ are non-compact but periodic.  The cylinder $\widehat{S}(0)$ is a Riemannian cover of $S_0$, and for each $t > 0$, $\widehat{S}(t)$ is a Riemannian cover of $(S_0,g(t))$, the Ricci flow initialized by $S_0$.

Moreover, for all $\xi \in \mathbb{R}$, $lim_{t \rightarrow \infty} f(\xi,t) = R_{\infty}$ where
\begin{equation} \nonumber
\left( R_{\infty} \right)^2 = \frac{1}{Q} \int_0^Q f_0(\xi)^2 \mbox{ d}\xi =
\frac{\displaystyle \int_0^P f_0(v) \sqrt{f_0^{\prime}(v)^2 + h_0^{\prime}(v)^2} \mbox{ d}v}{\displaystyle \int_0^P \left( \frac{1}{f_0(v)} \right) \sqrt{f_0^{\prime}(v)^2 + h_0^{\prime}(v)^2} \mbox{ d}v} .
\end{equation}
\end{theorem}
\begin{proof}
Let $(\xi, \theta)$ be isothermal coordinates on $S_0$ as in \eqref{xidefinition}, and let
\begin{equation} \nonumber
\widehat{S}_0 = \left\{ \left( f_0(\xi) \cos \theta, f_0(\xi) \sin \theta, \widehat{h}_0(\xi) \right): 0 \leq \theta \leq 2\pi, -\infty < \xi < \infty \right\}
\end{equation}
where $\widehat{h}_0(\xi) = \int \left| h_0^{\prime}(\xi) \right| \mbox{ d}\xi$ is a monotone function of $\xi$.

Figure 1 illustrates the geometry of the construction of the Riemannian cover $\widehat{S}_0$ of $S_0$ in the case when $S_0$ is generated by revolving a figure eight (its ``profile curve'') around the $z$-axis.  Starting at height $Z_S$, the profile curve is traversed (bold curve) with increasing $z$ until a local maximum is reached at height $Z_{C_1}$ (see Figure 1 - left side) is reached.  At this point, the untraversed portion of the curve is reflected vertically through the line $z=Z_{C_1}$ (see Figure 1 - right side).  From here, the tracing of the curve continues until another local maximum is reached at height $Z_{C_2}$.  The untraversed portion of the curve is again reflected vertically; this time, through the line $z=Z_{C_2}$.  This process continues until the profile curve has been ``unwound" to produce a monotonically non-decreasing curve, which is then extended periodically in both directions.

\begin{figure}[ht!]
\centerline{
\includegraphics[width=4.5in]{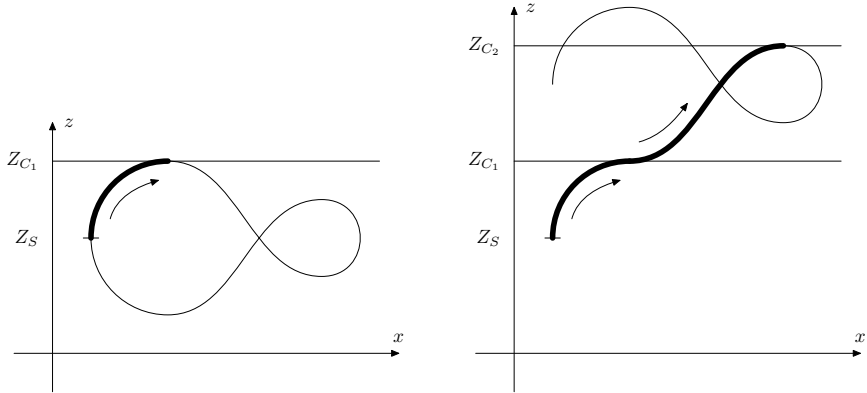}
}
\caption{The construction of the cover $\widehat{S}_0$ of $S_0$.}
\end{figure}

Part (b) of Theorem~\ref{easynopolecase} guarantees a global extrinsic representation of a Ricci flow $(\widehat{S}_0, \widehat{g}(t))$ initialized by $\widehat{S}_0$.  It is comprised of embedded surfaces $\widehat{S}(t)$ parameterized as in \eqref{constructedextrinsicflow}, where $\widehat{S}(0) = \widehat{S}_0$ and $u(\xi,t) = f(\xi,t)^2$ is the unique solution of \eqref{intrinsiclogdiffusion} specified by Lemma~\ref{existencetheorem} for the initial function $u_0(\xi) = f_0(\xi)^2$.  Because $f_0(\xi)$ has period $Q$, $u(\xi + Q,t)$ also qualifies as this unique solution, so that for any $t > 0$, $u(\xi,t) = u(\xi+Q,t)$ for all $\xi \in \mathbb{R}$.  The $L^1$ mass of $u(\xi,t)$ over one period is conserved since for $t > 0$, $(\mbox{d}/\mbox{d}t) \int_0^Q u(\xi,t) \mbox{ d}\xi = \int_0^Q u_t(\xi,t) \mbox{ d}\xi = \int_0^Q [\log u(\xi,t)]_{\xi \xi} \mbox{ d}\xi = [\log u(\xi,t)]_{\xi}(Q,t) - [\log u(\xi,t)]_{\xi}(0,t) = 0$.

Let $g_E$ be the metric induced on $S_0$ by the ambient Euclidean metric.  Because $(S_0, g_E)$ is a quotient of $(\widehat{S}_0, g_E)$  by the discrete group of isometries $G = \left\{ (\xi, \theta) \rightarrow (\xi + n\cdot Q, \theta) : n \in \mathbb{Z} \right\}$, and because the Ricci flow in question preserves the isometry group $G$ (due to the fact that $u(\xi,t)$ has period $Q$ in $\xi$), the quotient by $G$ produces, for each $t > 0$, a Riemannian covering $c(t): (\widehat{S}_0, \widehat{g}(t)) \rightarrow (S_0, g(t))$ \cite{Andrews2011}.  For each $t \geq 0$, composing the covering map $c(t)$ with the corresponding isometry $i(t): (\widehat{S}(t), g_E(t)) \rightarrow (\widehat{S}_0, \widehat{g}(t))$ yields the required local isometry
$c(t) \circ i(t) : (\widehat{S}(t), g_E(t)) \rightarrow (S_0, g(t))$.

That $f(\xi,t)$ approaches a limiting radius as $t \rightarrow \infty$ follows from the estimate \eqref{estimate}:  $(1/u)_{\xi \xi} \leq (1/t)$.  Writing out the second derivative, rearranging, and integrating over one period yields
\begin{equation} \label{firststep}
 2 \int_0^Q u_{\xi}^2 \mbox{ d}\xi \leq \frac{1}{t} \int_0^Q u^3 \mbox{ d} \xi + \int_0^Q u u_{\xi \xi} \mbox{ d} \xi.
\end{equation}
Integrating by parts in the last term of \eqref{firststep} then yields
\begin{eqnarray}
 2 \int_0^Q u_{\xi}^2 \mbox{ d}\xi \leq \frac{1}{t} & \leq & \int_0^Q u^3 \mbox{ d} \xi - \int_0^Q u_{\xi}^2 \mbox{ d} \xi   \nonumber \\
\int_0^Q u_{\xi}^2 \mbox{ d}\xi & \leq & \frac{1}{3t} \int_0^Q u^3 \mbox{ d} \xi. \label{secondstep}
\end{eqnarray}
Let $(R_{\infty})^2 = (1/Q)\int_0^Q u(\xi,t) \mbox{ d}\xi = (1/Q)\int_0^Q u_0(\xi)\mbox{ d}\xi$.  Then the Poincar\'{e} inequality and \eqref{secondstep} yield
\begin{equation} \nonumber
\left| u(\xi,t) - (R_{\infty})^2 \right| \leq Q \left[ \int_0^Q u_{\xi}^2 \mbox{ d}\xi \right]^{1/2} \leq Q \left[ \frac{1}{3t} \int_0^Q u^3 \mbox{ d}\xi \right]^{1/2}
\leq \frac{Q^{3/2} M^3}{\sqrt{3t}}
\end{equation}
for all $\xi \in \mathbb{R}$ and all $t > 0$, where $M = \sup_{\xi \in \mathbb{R}} f_0(\xi)$.
\end{proof}

\begin{remark}
In general, the $z$-period $Z(t)$ of the surface $\widehat{S}(t)$ in Theorem~\ref{periodicinitialcurve} changes with time $t$, but $Z(t) \rightarrow R_{\infty}Q$ as $t \rightarrow \infty$.
\end{remark}

\begin{example}[{\bf The standard torus}]
Let $S_0$ be a standard torus parameterized by \eqref{initialsurface} with $f_0(v) = a + b \cos v$ and $h_0(v) = b\sin v$ where $a > b > 0$,
and consider the Ricci flow initialized by this torus.  The limiting radius $R_{\infty}$ of Theorem~\ref{periodicinitialcurve} is an average of the outer
radius $R_O = a + b$ and inner radius $R_I = a - b$ of the torus given by
$$(R_{\infty})^2 = \frac{\int_0^{2\pi} a + b \cos v \mbox{ d}v}{\int_0^{2\pi} 1/\left( a + b \cos v \right) \mbox{ d}v}
= a \sqrt{a^2 - b^2}
= \left( \frac{R_O + R_I}{2} \right) \sqrt{R_O R_I}. $$
For the evolving Riemannian cover $\widehat{S}(t)$ of the Ricci flow $(S_0, g(t))$ initialized by $S_0$ constructed in Theorem~\ref{periodicinitialcurve},
the limiting $z$-period, $\lim_{t \rightarrow \infty} Z(t)$, is given by
$$ R_{\infty} Q = \sqrt{\left( a \sqrt{a^2 - b^2} \right)} \left( \int_0^{2\pi} \frac{b}{a + b \cos (v)} \mbox{ d}v \right) =
   2 \pi b \sqrt{\frac{a}{\sqrt{a^2 - b^2}}} . $$
Note that this limiting $z$-period is always greater than the circumference $2 \pi b$ of the generating circle for the torus.
\end{example}

\section{Compact evolution of tori}

It is possible to modify the extrinsic representations constructed in Theorem~\ref{periodicinitialcurve}, which are comprised of smooth but non-compact surfaces embedded in $\mathbb{R}^3$,
to create alternate representations comprised of compact but non-smooth surfaces immersed in $\mathbb{R}^3$.

\begin{theorem} \label{creasetheorem}
Let $S_0$ be a toroidal surface of revolution immersed in $\mathbb{R}^3$ by the parametrization \eqref{initialsurface}.  Let $Q$ be the common period of $f_0(\xi)$ and $h_0(\xi)$ in isothermal coordinates $(\xi, \theta)$.  Let $\widehat{S}(t)$ be the local extrinsic representation of the Ricci flow initialized by $S_0$ given by Theorem~\ref{periodicinitialcurve}, with parametrization $\left( f(\xi,t) \cos \theta, f(\xi,t) \sin \theta, \widehat{h}(\xi,t) \right)$ for $0 \leq \theta \leq 2\pi$ and $-\infty < \xi < \infty$.  Recall that for all $t \geq 0$, $f(\xi,t)$ has period $Q$
and $\widehat{h}(\xi,t)$ is a monotone non-decreasing function such that $\widehat{h}(\xi+Q,t)= \widehat{h}(\xi,t)+Z(t)$ for a positive function $Z(t)$, which is the $z$-period of $\widehat{S}(t)$.

Choose an arbitrary starting height $z_0$.  For each $t \geq 0$, choose $\xi_0(t)$ such that $\widehat{h}(\xi_0(t),t)= z_0$.  Choose $\xi_1(t)> \xi_0(t)$ such that $\widehat{h}(\xi_1(t),t) = z_0 + Z(t)/2$, and define the following functions:
\begin{eqnarray} \label{crease}
f(\xi,t) & := & f(\xi,t) \nonumber \\
h_c (\xi,t) & := & \begin{cases}
\widehat{h}(\xi,t), & \mbox{for } \xi \in [\xi_{0}(t),\xi_{1}(t)]\\
2z_{0}+Z(t)-\widehat{h}(\xi,t), & \mbox{for }\xi\in[\xi_{1}(t),\xi_{0}(t) + Q]. \nonumber
\end{cases}
\end{eqnarray}
Then for each $t \geq 0$, the surface of revolution
\small
\begin{equation} \nonumber
S_c(t)  =  \left\{ \left( f(\xi,t) \cos \theta, f(\xi,t) \sin \theta, h_c(\xi,t) \right) :  \xi_0(t) < \xi < \xi_0(t) + Q, 0 \leq \theta \leq 2\pi  \right\}
\end{equation}
\normalsize
is compact, continuous, and non-smooth only at the circles corresponding to $\xi=\xi_0(t)$ and $\xi=\xi_1(t)$, where it has creases.  For $\xi \in (\xi_0(t),\xi_1(t))$ and $\xi \in (\xi_1(t),\xi_0(t) + Q)$, the surfaces $S_c(t)$ comprise a local extrinsic representation  in $\mathbb{R}^3$ of the Ricci flow initialized by $S_0$.
\end{theorem}
\begin{proof}
The periodicity of $f(\xi,t)$ with period $Q$, and the property $\widehat{h}(\xi+Q,t)=\widehat{h}(\xi,t)+Z(t)$ from Theorem~\ref{periodicinitialcurve}, imply that for all $t \geq 0$, $S_c(t)$ is  generated as a surface of revolution by a closed curve with singularities only at points corresponding to $\xi_0(t)$ and and $\xi_1(t)$. That the smooth parts of $S_c(t)$ are local extrinsic representations of the Ricci flow follows from Theorem~\ref{periodicinitialcurve} as well.  Since the metric $g(t)$ on $S_c(t)$ depends on $\left(\partial h_c / \partial \xi  \right)^2$, the metric will be the same for $S_c(t)$ as for $\widehat{S}(t)$ at corresponding points.  Also, the metric can be extended smoothly to the creases by that fact as well, by \eqref{isothermalmetricflow}.
\end{proof}

\begin{remark}
In the proof of Theorem~\ref{creasetheorem}, $\xi_0(t)$ or $\xi_1(t)$ may not be uniquely determined if $\widehat{h}(\xi,t)$ is not strictly monotone as a function of $\xi$ when $\widehat{h}(\xi,t)=z_0$ or $\widehat{h}(\xi,t)=z_0+Z(t)/2$, but the surfaces $S_c(t)$ constructed in this theorem are independent of any such choices.
\end{remark}

\begin{remark}
The initial surface of the flow $S_c(t)$ constructed in Theorem~\ref{creasetheorem} may not be the same as the original toroidal surface $S_0$, but it is globally isometric to that surface.  If the height function $h_0$ of the original toroidal surface $S_0$ has only two local extrema, and if they correspond to the choice of $\xi_0(0)$ and the half-way point $\xi_1(0)$, then $\left.S_c(t)\right|_{t=0}$ will be identical to $S_0$ up to an ambient motion.  If there are other local extrema of the height, then it is possible to ``unwind'' $\widehat{h}$ at these points as well to re-construct the original surface.

In the case of a standard embedded torus of revolution generated by revolving a circle around the axis, when $\xi_0(0)$ and $\xi_1(0)$ are taken to correspond to the minimum and maximum heights of the circle, the family $S_c(t)$ constructed by this theorem will correspond, at $t=0$, to the original torus, and for any $t>0$ there will be two singular creases at the top and bottom of $S_c(t)$; see Figure 2, where the creases are indicated by the large bold faced dots.  As $t\to\infty$, $S_c(t)$ approaches a double-covered circular cylinder whose height is half of the limiting $z$-period of the corresponding cover $\widehat{S}(t)$ of the torus, or $R_{\infty}Q/2$.
\end{remark}

\begin{figure}[ht!]
\centerline{
\includegraphics[height=2in]{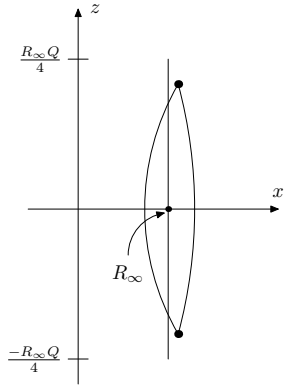}
}
\caption{The compact evolution of the standard torus}
\end{figure}

\section{Questions}

An extrinsic representation exists for the unique Ricci flow initialized by any smoothly immersed surface of revolution that is compact.  An extrinsic representation exists for at least one Ricci flow initialized by any smoothly immersed non-compact but complete surface of revolution whose profile curve satisfies certain regularity conditions and maintains a distance from the axis of revolution that is both bounded and bounded away from zero.  Each of these extrinsic representations that we have identified exists for the entire flow.  But many questions remain.

Do extrinsic representations in $\mathbb{R}^3$ exist for Ricci flows initialized by other smoothly immersed non-compact but complete surfaces of revolution, such as unbounded surfaces of revolution having only one pole?  For any non-negative initial function $u_0 \in L^1_{\rm loc}(\mathbb{R})$, the initial value problem \eqref{intrinsiclogdiffusion} has a multitude of solutions $u(\xi,t) \in C^{\infty}(\mathbb{R} \times (0,\infty))$~\cite{Rodriguez1995}.  A Ricci flow initialized by a given smoothly immersed non-compact but complete surface of revolution has an extrinsic representation in $\mathbb{R}^3$ if and only if there exists a corresponding solution of the initial value problem \eqref{intrinsiclogdiffusion} determined by the initial surface which satisfies conditions~(1) and~(2) of Theorem~\ref{representationtheorem}.   Condition~(1) is a natural condition fixing the flux associated with the logarithmic diffusion equation at ``$\pm \infty$''.  Condition~(2) is not unreasonable because, letting $w = \log u(\xi,t)$ where $u(\xi,t) = f(\xi,t)^2$, and letting $\phi = \dfrac{f_{\xi}}{f}$, it is easy to see that $\phi(\xi,t)$ satisfies the parabolic equation
$$ \phi_t = \left({\rm e}^{-w}\right) \phi_{\xi \xi} - \left({\rm e}^{-w}\right) \phi \phi_{\xi}. $$
Since $\sup_{\xi \in \mathbb{R}} \left| \phi(\xi,0) \right| = \sup_{\xi \in \mathbb{R}} \left| \dfrac{f_0^{\prime}(\xi)}{f_0(\xi)} \right| \leq 1$ by the immersability of the initial surface $S_0$,  condition~(2) will hold if an appropriate maximum principle holds for this parabolic equation.  It would be of great interest to know whether or not these two conditions can be satisfied for any Ricci flow initialized by any smoothly immersed non-compact but complete surface of revolution, and if not, to be able to identify specific counterexamples.  If the Ricci flow initialized by an immersed surface of revolution is not unique, is it possible for some of the flows to have extrinsic representations in $\mathbb{R}^3$ while others do not?  Can a Ricci flow that exists for times $0 < t < T_1$ have an extrinsic representation in $\mathbb{R}^3$ only for times $0 < t < T_2$ where $T_2 < T_1$?

Finally, we point out that while our discussion has been from an extrinsic point of view, it raises some closely related intrinsic issues.  For example, in order for a Ricci flow initialized by an abstract surface of revolution $(S, g_0)$ to have an extrinsic representation in $\mathbb{R}^3$, $(S, g_0)$ must be smoothly embeddable in $\mathbb{R}^3$.  When does such an embedding exist?  This issue has been addressed both classically and with modern tools in the case that $S$ is a sphere (see for example M.\ Engman~\cite{Engman1988, Engman2006}), and more recently has been addressed in the case that $S$ is a torus (see Q.\ Han and F.\ Lin~\cite{Han2008}).  We also note that the close relationship between the scalar logarithmic diffusion equation in one space variable and Ricci flow on abstract surfaces of revolution has been explored from an intrinsic point of view, in the context of Ricci flows on the plane initialized by radially symmetric metrics, by J.\ L.\ V{\'a}zquez, J.\ R.\ Esteban, and A.\ Rodr{\'{\i}}guez~\cite{Vazquez1996}.

\subsection*{Acknowledgments}
The authors wish to thank Linghai Zhang for his help pointing the way to the solution of the  differential equation for the Ricci flow of a surface of revolution, and we wish to thank Huai-Dong Cao for his many helpful discussions about these results.

\bibliographystyle{amsplain}

\bibliography{bibliographyfileforCollDoddJohnson}

\bigskip

\noindent Vincent Coll, Department of Mathematics, Lehigh University, 27 Memorial Drive West, Bethlehem PA 18015
(\verb+vecjr@iconcepts-inc.com+)

\bigskip

\noindent Jeff Dodd, Mathematical, Computing and Information Sciences Department, Jacksonville State University,
700 Pelham Road North, Jacksonville AL 36265 (\verb+jdodd@jsu.edu+)

\bigskip

\noindent David L.\ Johnson, Department of Mathematics, Lehigh University, 27 Memorial Drive West, Bethlehem PA 18015
(\verb+david.johnson@lehigh.edu+)

\end{document}